\documentclass[11pt,twoside,english]{article}
\usepackage{amsmath}
\usepackage{amsfonts}
\usepackage{mathrsfs}
\usepackage{color}
\usepackage{ amsmath, amsfonts, amssymb, amsthm, amscd}
\usepackage[T1]{fontenc}
\usepackage[english]{babel}
\usepackage[noinfoline]{imsart}

\newtheorem{theorem}{Theorem}
\newtheorem{lemma}{Lemma}
\newtheorem{proposition}{Proposition}

\newtheorem{example}{Example}
\newtheorem{definition}{Definition}

\newtheorem{remark}{Remark}

\newcommand{\cO}{\ensuremath{\mathcal O}}

\newcommand{\bbN}{{\ensuremath{\mathbb N}} }

\newcommand{\bbR}{{\ensuremath{\mathbb R}} }

\newcommand{\be}{\begin{equation}}
\newcommand{\ee}{\end{equation}}
\newcommand{\beq}{\begin{eqnarray}}
\newcommand{\eeq}{\end{eqnarray}}

\newcommand{\R}{\mathbb{R}}

\newcommand{\ced}{\end{proof}}

\begin{document}
\begin{frontmatter}
\title{The Obstacle Problem for Quasilinear Stochastic PDEs with Degenerate Operator}
\date{}
\runtitle{}
\author{\fnms{Xue}
 \snm{YANG}\corref{}\ead[label=e1]{xyang2013@tju.edu.cn}}
\address{Tianjin University
\\\printead{e1}}
\author{\fnms{Jing}
 \snm{ZHANG}\corref{}\ead[label=e2]{zhang\_jing@fudan.edu.cn}}
\address{Fudan University
\\\printead{e2}}

\runauthor{X. Yang and J. Zhang}

\begin{abstract}
We prove the existence and uniqueness of
solution of quasilinear stochastic partial differential equations with obstacle (OSPDEs in short) in degenerate case.
Using De Giorgi's iteration, we deduce the $L^p-$estimates for the time-space uniform norm of weak solutions.
\end{abstract}


\end{frontmatter}

\section{Introduction}

We consider the following stochastic partial differential equation (SPDE in short)
\begin{equation}
\label{SPDE}
\begin{split}
du_t(x)&=\,\partial_i\Big(\sum_{j=1}^{d}a_{ij}(x)\partial_j u_{t}(x)+\sum_{j=1}^d \sigma_{ij} g_j(t,x,u_t(x),\sigma^T\nabla u_t(x))\Big)dt\\
&+\,f(t,x,u_t(x),\sigma^T\nabla u_t(x))dt
+ \sum_{j=1}^{+\infty}h_j(t,x,u_t(x),\sigma^T\nabla u_t(x))dB^{j}_t,
\end{split}
\end{equation}
where $a=(a_{ij})=\sigma \sigma^T$ is a degenerate symmetric bounded measurable matrix which defines a second order operator on $\mathcal{O}\subset \mathbb{R}^d$ with Dirichlet boundary condition. The initial condition is given as $u_0=\xi$, an $L^2(\mathcal{O})-$valued random
variable, and $f$, $g=(g_1,...,g_d)$ and $h=(h_1,...h_i,...)$ are non-linear random functions. The obstacle $S$ is a given real-valued process defined on $\Omega\times[0,T]\times\mathcal{O}$, and we study the obstacle problem for SPDE which means that we try to find a solution for $(\ref{SPDE})$ satisfying the condition "$u\geq S$" where the obstacle $S$ is regular in some sense.\\

The obstacle problems for SPDEs have been widely studied in recent years. Nualart and Pardoux \cite{NualartPardoux} firstly proved the existence and uniqueness of  solution for obstacle problem for heat equation driven by a space-time white noise with the diffusion matrix $a=(a_{ij})=I$, and then by Donati-Martin and Pardoux \cite{DonatiPardoux} for the general drift and diffusion coefficients.  Various properties of  the solutions were studied later in \cite{DMZ}, \cite{XuZhang}, \cite{Zhang} etc..
As backward stochastic differential equations (BSDEs in short) were rapidly developed since \cite{PardouxPeng}, \cite{PardouxPeng92}, the obstacle problem for nonlinear partial differential equations (PDEs in short) with more general coefficients and the properties of the solutions were studied via BSDEs.  In \cite{EKPPQ}, the reflected BSDE was introduced, in which there was an increasing process $K$ only having increments on the set $\{u=S\}$ to force the solution to stay above the obstacle $S$. They showed that the reflected BSDEs provided a probabilistic interpretation for the unique viscosity solution of the obstacle problem for PDEs.  Bally, El-Karoui et al. \cite{BCEF} constructed the solutions of reflected BSDEs by maximal principle but by applying the classic penalization method used in  \cite{EKPPQ} and \cite{BensoussanLions78}, and explained the relationship between reflected BSDEs and variational inequalities.\\

The purpose of this paper is two-fold. Firstly, we prove the existence and uniqueness of solution for the obstacle problem of an SPDE in which the diffusion matrix $a$ is degenerate. This work is mainly motivated by \cite{MatoussiStoica} and \cite{DMZ12} in which the diffusion matrix is uniformly elliptic. Interpreting the solution by using backward doubly stochastic differential equations, Matoussi and Stoica  \cite{MatoussiStoica} proved an existence and uniqueness result for the obstacle problem of quasilinear SPDEs on the whole space $\mathbb{R}^d$ and driven by a finite dimensional Brownian motion. In \cite{DMZ12}, Denis, Matoussi and Zhang studied the OSPDE  with null Dirichlet condition on an open domain in $\mathbb{R}^d$ and driven by an infinite dimensional Brownian motion. Their method was based on the techniques of parabolic potential theory developed by M. Pierre (\cite{Pierre}, \cite{PIERRE}). The key point was to construct a solution which admitted a quasi-continuous version defined outside a polar set and the regular measures which  are not absolutely continuous w.r.t. the Lebesgue measure in general, do not charge polar sets. In our paper, the most difficulty consisted in the degeneracy of the second order operator, which means the sequence of penalized processes will not converge in the space $L^2(\Omega\times[0,T], H^1_0(\mathcal{O}))$. However, as long as the divergence term in the range of the diffusion matrix, we are able to control the penalization sequence in some proper space, and then the existence of the solution is obtained.\\

Secondly,  we deduce the $L^p$-estimates for the uniform time space norm of weak solutions.  In \cite{DMZ13}, the authors established the $L^p$-estimates for the uniform norm of the solution and proved a maximum principle for local solutions of quasilinear stochastic PDEs with obstacle by the method of stochastic Moser's iteration scheme. Inspired by the  version of stochastic De Giorgi iteration used in \cite{Qiu}, we obtain the $L^p$-estimates for the uniform norm of the solution under suitable integrability conditions on the coefficients. Since the second order elliptic operator is degenerate, the solution is not in the second order Hilbert spaces as discussed in \cite{DMZ12}, but it follows that it is in a fractional order Hilbert space which is also a contribution of our work.\\

The paper is organized as follows: in the second section, we set the assumptions then we
introduce in the third section the notion of regular measure associated to parabolic potentials. The quasi-continuity of the solution for SPDE without obstacle  is proved in the forth section. The fifth section is devoted to proving the existence and uniqueness of solution and to establishing It\^o's formula and the comparison theorem. In the last section, we  obtain the $L^p$-estimates for the time-space uniform norm of weak solutions.

 \section{Preliminaries}

We consider a sequence $((B^i(t))_{t\geq0})_{i\in\mathbb{N}^*}$ of
independent Brownian motions defined on a standard filtered
probability space $(\Omega,\mathcal{F},(\mathcal{F}_t)_{t\geq0},P)$
satisfying the usual conditions.

Let $\mathcal{O}\subset \bbR^d$ be an open domain  and
$L^2(\mathcal{O})$ the set of square integrable functions with
respect to the Lebesgue measure on $\mathcal{O}$, and it is a Hilbert
space equipped with the usual inner product and norm as follows
$$(u,v)=\int_\mathcal{O}u(x)v(x)dx,\qquad\| u\|=(\int_\mathcal{O}u^2(x)dx)^{1/2}\,.$$
For two vector valued functions $q=(q_1,\cdots,q_n)$, $p=(p_1,\cdots,p_n)$, where $q_i, p_i\in L^2(\mathcal{O}), i=1,...,n$,
we also use the notation $(p,q):=\int_\mathcal{O} \sum_{i=1}^{n}q_i(x)p_i(x)dx$ for simplicity.
\\Let $A$ be a symmetric second order differential operator, with domain $\mathcal{D}(A)$, given by
$$A:=-L:=-\sum_{i,j=1}^d\partial_i(a_{ij}\partial_j).$$
We assume that
$a=(a_{ij})_{1\leq i,j\leq d}$ is a measurable symmetric matrix defined on
$\mathcal{O}$ which satisfies the following degenerate elliptic condition:
$$
0\leq\sum_{ij}a_{ij}(x)\xi_i\xi_j\leq \lambda_0|\xi|^2,\ \  x\in \mathcal{O},\ \  \xi\in\mathbb{R}^d,
$$
and it has the form $a_{ij}=\sum_{k=1}^{d}\sigma_{ik}\sigma_{jk}$, where $x\rightarrow \sigma(x)=(\sigma_{ik}(x))$ is a bounded measurable map from $\mathcal{O}\subseteq\mathbb{R}^d$ to the set of $d\times d$ real matrices.

Let $(\mathfrak{F},\mathcal{E})$ be the Dirichlet form associated with the operator $A$ on $L^2(\mathcal{O})$,  defined as follows
$$
\mathcal{E}(u,u)=(a\nabla u, \nabla u  ),\ \  u\in \mathfrak{F},
$$
where $\mathfrak{F}$ the domain of Dirichlet form is the closure of $C_c^\infty(\mathcal{O})$ under the norm $\|\cdot\|^2_{\mathcal{E}_1}:=\mathcal{E}(\cdot,\cdot)+(\cdot,\cdot)$. Then $\mathfrak{F}$ is a Hilbert space with the norm $\|\cdot\|_{\mathcal{E}_1}$ and $\mathfrak{F}'$ denotes its dual space.

We consider the quasilinear stochastic partial differential equation
(\ref{SPDE}) with initial condition $u(0,\cdot)=\xi(\cdot)$ and Dirichlet
boundary condition $u(t,x)=0,\ \forall\ (t,x)\in
\bbR^+\times\partial\mathcal{O}$.

Assume that we have predictable random
functions
\begin{eqnarray*}&&f:\bbR^+\times\Omega\times\mathcal{O}\times
\bbR\times \bbR^d\rightarrow
\bbR,\\&&g=(g_1,...,g_d):\bbR^+\times\Omega\times\mathcal{O}\times \bbR\times
\bbR^d\rightarrow
\bbR^d,\\&&h=(h_1,...,h_i,...):\bbR^+\times\Omega\times\mathcal{O}\times
\bbR\times \bbR^d\rightarrow \bbR^{\mathbb{N}^*}.\end{eqnarray*}In the sequel,
$|\cdot|$ will always denote the underlying Euclidean or $l^2-$norm.
For
example$$|h(t,\omega,x,y,z)|^2=\sum_{i=1}^{+\infty}|h_i(t,\omega,x,y,z)|^2.$$

\noindent\textbf{Assumption (H):} There exist non-negative constants $C,\
\alpha,\ \beta$ such that for almost all $\omega$, the following
inequalities hold for all
$(t,x,y,z)\in\mathbb{R}^+\times\mathcal{O}\times\mathbb{R}\times\mathbb{R}^d$:
\begin{enumerate}
\item   $|f(t,\omega,x,y,z)-f(t,\omega,x,y',z')|\leq C(|y-y'|+|z-z'|),$
\item $(\sum_{i=1}^d|g_i(t,\omega,x,y,z)-g_i(t,\omega,x,y',z')|^2)^{\frac{1}{2}}\leq
C|y-y'|+\alpha|z-z'|,$
\item $|h(t,\omega,x,y,z)-h(t,\omega,x,y',z')|\leq
C|y-y'|+\beta|z-z'|,$
\item the contraction property: $2\alpha+\beta^2<1.$
\end{enumerate}
\noindent Moreover for simplicity, we fix a terminal time $T >0$, we assume that: \\[0.3cm]
 \textbf{Assumption (I):} $$\xi\in L^2(\Omega\times\mathcal{O})\ is\ an\
\mathcal{F}_0-measurable\ random\
variable$$$$f(\cdot,\cdot,\cdot,0,0):=f^0\in
L^2([0,T]\times\Omega\times\mathcal{O};\bbR)$$$$g(\cdot,\cdot,\cdot,0,0):=g^0=(g_1^0,...,g_d^0)\in
L^2([0,T]\times\Omega\times\mathcal{O};\bbR^d)$$$$h(\cdot,\cdot,\cdot,0,0):=h^0=(h_1^0,...,h_i^0,...)\in
L^2([0,T]\times\Omega\times\mathcal{O};\bbR^{\mathbb{N}^*}).$$

Now we introduce the notion of weak solution.\\
For simplicity, we fix the terminal time $T>0$. We denote by $\mathcal{H}_T$ the space of $\mathfrak{F}-$valued predictable continuous processes $(u_t)_{t\geq0}$ which satisfy
$$
E\sup_{t\in[0,T]}\|u_t\|^2+E\int_{0}^{T}\|\sigma^T \nabla u_t\|^2dt< \infty\,.
$$It is the space in which solutions exist. The space of test functions is denote by  $\mathcal{D}=\mathcal{C}%
_{c}^{\infty } (\R^+ )\otimes \mathcal{C}_c^2 (\cO )$, where $\mathcal{C}%
_{c}^{\infty } (\R^+ )$ is the space of all real valued infinitely
differentiable  functions with compact support in $\mathbb{R}^+$ and
$\mathcal{C}_c^2 (\cO )$ the set
of $C^2$-functions with compact support in $\cO$.\\
Heuristically, a pair $(u,\nu)$ is  a solution of the obstacle
problem for (\ref{SPDE}) if we have the followings:
\begin{definition}
A pair $(u,\nu)$ is said to be a solution of the obstacle problem for SPDE (\ref{SPDE}) if we have the followings:
\begin{enumerate}
\item $u\in\mathcal{H}_T$ and $u(t,x)\geq S(t,x)$, $dP\otimes dt\otimes dx-$a.e. and $u_0(x)=\xi$, $dP\otimes dx-$a.e.;
\item $\nu$ is a random measure the support of which is on $\{(t,x): u(t,x)=S_t(x)\}$;
\item the following relation holds almost surely, for all
$t\in[0,T]$ and $\forall\varphi\in\mathcal{D}$,
\begin{eqnarray*}&&\hspace{-1cm}(u_t,\varphi_t)-(\xi,\varphi_0)-\int_0^t(u_s,\partial_s\varphi_s)ds+\int_0^t(\sigma^T \nabla u_s, \sigma^{T}\nabla \varphi_s)ds+\int_0^t(\sigma^T\nabla\varphi_s, g_s(u_s,\sigma^T\nabla u_s))ds\nonumber\\
&&\hspace{-1cm}=\int_0^t(\varphi, f_s(u_s,\sigma^T\nabla u_s))ds+\sum_{j=1}^{+\infty}\int_0^t(\varphi_s, h^j_s(u_s,\sigma^T\nabla u_s))dB^j_s+\int_0^t\int_{\mathcal{O}}\varphi_s(x)\nu(dx,ds)\,;
\end{eqnarray*}
\item $$\int_0^T\int_{\mathcal{O}}(u(s,x)-S(s,x))\nu(dx,ds)=0,\ \ a.s..$$
\end{enumerate}
\end{definition}
But, the random measure which in some sense obliges the solution to stay above the barrier is a local time. Hence, in general, it is not
absolutely continuous w.r.t Lebesgue measure. As a consequence, for
example, the condition
$$\int_0^T\int_\mathcal{O}(u(s,x)-S(s,x))\nu(dxds)=0$$ makes no
sense in some way. Hence we need to consider precise version of $u$ and $S$
defined $\nu-$almost surely. 

In order to tackle this difficulty, we introduce  in the next section the notions of parabolic capacity on
$[0,T]\times\mathcal{O}$ and quasi-continuous version of functions
introduced by Michel Pierre (see for example \cite{Pierre,PIERRE}). \\

To end this section, we give an example to show the reason why the divergence term is in the form of $div(\sigma g)$, in the case $a=(a_{ij})$ is degenerate.

\begin{example}
We consider the 1-dimensional case. Set $\mathcal{O}=(0,\pi)$, and let the second order derivative vanished, i.e. $a=0$.
$e^k(x)=\sqrt{\frac{2}{\pi}}sin kx$ is a complete standard orthogonal basis in $L^2(0,\pi)$ and also an orthogonal basis in $H^1_0(0,\pi)$.

\vspace{2mm}
Set the coefficients in SPDE (\ref{SPDE}) as follows, $f=0 $, $g(t,x)=\sum_k \frac{1}{k} (\sqrt{\frac{2}{\pi}} coskx)$.  Since $\{\sqrt{\frac{2}{\pi}} coskx\}$ is also a standard orthogonal basis in $L^2(0,\pi)$, it follows $g\in L^2([0,T]\times\mathcal{O})$. It is obvious that  $g$ is not in the range of matrix $a=(a_{ij})=0$ in this case.

Set the function $h(\omega,t,x)\in L^2(\Omega\times[0,T]\times\mathcal{O})$, then there is the decomposition $h(\omega,t,x)=\sum_{n=1}^{\infty}h^n_t(\omega) e^n(x)$. 

\vspace{2mm}
Suppose the vector $u^{N}:=(u^{1,N},..., u^{N,N})$ satisfies the N-dimensional stochastic differential equations, for $1\leq n\leq N$,

\begin{eqnarray*}
du^{n,N}_t&=&-(g(t,x), \partial e^n)_{L^2}dt+h^n_tdB^n_t\nonumber\\
&=&-(g(t,x), n\sqrt{\frac{2}{\pi}}cosnx)_{L^2}dt+h^n_tdB^n_t\nonumber\\
&=&-dt+h^n_tdB^n_t,
\end{eqnarray*}
where the initial condition $u^{n,N}_0=0$.

Set $w^N(t,x)=\sum_{n=1}^{N}u^{n,N}_t e^n(x)\in L^2(\Omega\times [0,T], H_{0}^{1})$,

\begin{eqnarray*}
E\|w^N(t)\|^2&=&E\int_{0}^{t}\sum_{n=1}^{N}\|h^n_s\|^2ds-2E\int_{0}^{t}\sum_{n=1}^{N}(u^{n,N}_s,1)ds.
\end{eqnarray*}
But the right hand side can  not be convergent, which means the sequence $w^N$ does not have limit. Therefore, in this extremely degenerate example, the SPDE
\begin{equation*}
du_t=\partial g(t,x)dt+\sum_{n}h^n_tdB^n_t
\end{equation*}
does not have solution even in $L^2(\Omega\times[0,T]\times(0,\pi))$, because $g$ does not equal to 0 when the diffusion matrix vanishes.  It is an example showing that the divergence term not in the range of diffusion matrix may lead to the SPDE unsolved.
\end{example}

\section{Parabolic Potential Analysis}{\label{capacity}}
\subsection{Parabolic capacity and potentials}
In this section we will recall some important definitions and
results concerning the obstacle problem for parabolic PDE in
\cite{Pierre} and \cite{PIERRE}.
\\$\mathcal{K}$ denotes $L^\infty(0,T;L^2(\mathcal{O}))\cap
L^2(0,T;\mathfrak{F})$ equipped with the norm:
\begin{eqnarray*}\parallel
v\parallel^2_\mathcal{K}&=&\parallel
v\parallel^2_{L^\infty(0,T;L^2(\mathcal{O}))}+\parallel
v\parallel^2_{L^2(0,T;\mathfrak{F})}\\
&=&\sup_{t\in[0,T[}\parallel v_t\parallel^2 +\int_0^T \mathcal{E}_1(v_t)\, dt
.\end{eqnarray*} $\mathcal{C}$ denotes the space of continuous
functions on compact support in $[0,T[\times\mathcal{O}$ and
finally:
$$\mathcal {W}=\{\varphi\in L^2(0,T;\mathfrak{F});\ \frac{\partial\varphi}{\partial t}\in
L^2(0,T;\mathfrak{F}')\}, $$ endowed with the
norm$\parallel\varphi\parallel^2_{\mathcal {W}}=\parallel
\varphi\parallel^2_{L^2(0,T;\mathfrak{F})}+\parallel\displaystyle\frac{\partial
\varphi}{\partial t}\parallel^2_{L^2(0,T;\mathfrak{F}')}$. \\
It is known (see \cite{LionsMagenes}) that $\mathcal{W}$ is
continuously embedded in $C([0,T]; L^2 (\cO))$, the set of $L^2 (\cO
)$-valued continuous functions on $[0,T]$. So without ambiguity, we
will also consider
$\mathcal{W}_T=\{\varphi\in\mathcal{W};\varphi(T)=0\}$,
$\mathcal{W}^+=\{\varphi\in\mathcal{W};\varphi\geq0\}$,
$\mathcal{W}_T^+=\mathcal{W}_T\cap\mathcal{W}^+$.\\We now introduce the notion of parabolic potentials and regular measures to define the parabolic capacity.
\begin{definition}
An element $v\in \mathcal{K}$ is said to be a parabolic potential if it satisfies:
$$ \forall\varphi\in\mathcal{W}_T^+,\
\int_0^T-(\frac{\partial\varphi_t}{\partial
t},v_t)dt+\int_0^T\mathcal{E}(\varphi_t, v_t)dt\geq0.$$
We denote by $\mathcal{P}$ the set of all parabolic potentials.
\end{definition}
\noindent The next representation property is  crucial:
\begin{proposition}(Proposition 1.1 in \cite{PIERRE})\label{presentation}
Let $v\in\mathcal{P}$, then there exists a unique positive Radon
measure on $[0,T[\times\mathcal{O}$, denoted by $\nu^v$, such that:
$$\forall\varphi\in\mathcal{W}_T\cap\mathcal{C},\ \int_0^T(-\frac{\partial\varphi_t}{\partial t},v_t)dt+\int_0^T\mathcal{E}(\varphi_t, v_t)dt=\int_0^T\int_\mathcal{O}\varphi(t,x)d\nu^v.$$
Moreover, $v$ admits a right-continuous (resp. left-continuous)
version $\hat{v} \ (\makebox{resp. } \bar{v}): [0,T]\mapsto L^2
(\cO)$ .\\
Such a Radon measure, $\nu^v$ is called {\bf a regular measure} and we write:
$$ \nu^v =\frac{\partial v}{\partial t}+Av .$$
\end{proposition}

\begin{definition}
Let $K\subset [0,T[\times\mathcal{O}$ be compact, $v\in\mathcal{P}$
is said to be  \textit{$\nu-$superior} than 1 on $K$, if there
exists a sequence $v_n\in\mathcal{P}$ with $v_n\geq1\ a.e.$ on a
neighborhood of $K$ converging to $v$ in
$L^2([0,T];\mathfrak{F})$.
\end{definition}
We denote:$$\mathscr{S}_K=\{v\in\mathcal{P};\ v\ is\ \nu-superior\
than\ 1\ on\ K\}.$$
\begin{proposition}(Proposition 2.1 in \cite{PIERRE})
Let $K\subset [0,T[\times\mathcal{O}$ be compact, then $\mathscr{S}_K$
admits a smallest $v_K\in\mathcal{P}$ and the measure $\nu^v_K$
whose support is in $K$ satisfies
$$\int_0^T\int_\mathcal{O}d\nu^v_K=\inf_{v\in\mathcal{P}}\{\int_0^T\int_\mathcal{O}d\nu^v;\ v\in\mathscr{S}_K\}.$$
\end{proposition}
\begin{definition}(Parabolic Capacity)\begin{itemize}
                                        \item Let $K\subset [0,T[\times\mathcal{O}$ be compact, we define
$cap(K)=\int_0^T\int_\mathcal{O}d\nu^v_K$;
                                        \item let $O\subset
[0,T[\times\mathcal{O}$ be open, we define $cap(O)=\sup\{cap(K);\
K\subset O\ compact\}$;
                                        \item   for any Borelian
$E\subset [0,T[\times\mathcal{O}$, we define $cap(E)=\inf\{cap(O);\
O\supset E\ open\}$.
                                      \end{itemize}

\end{definition}
\begin{definition}A property is said to hold quasi-everywhere (in short q.e.)
if it holds outside a set of null capacity.
\end{definition}
\begin{definition}(Quasi-continuous)

\noindent A function $u:[0,T[\times\mathcal{O}\rightarrow\mathbb{R}$
 is called quasi-continuous, if there exists a decreasing sequence of open
subsets $O_n$ of $[0,T[\times\mathcal{O}$ with: \begin{enumerate}
                        \item for all $n$, the restriction of $u_n$ to the complement of $O_n$ is
continuous;
                                       \item $\lim_{n\rightarrow+\infty}cap\;(O_n)=0$.
                                     \end{enumerate}
We say that $u$ admits a quasi-continuous version, if there exists
$\tilde{u}$ quasi-continuous  such that $\tilde{u}=u\ a.e.$.
\end{definition}

\subsection{Applications to PDE's with obstacle}
For any function $\psi:[0,T[\times\mathcal{O}\rightarrow \bbR$ and
$u_0\in L^2(\mathcal{O})$, following  M. Pierre
\cite{Pierre,PIERRE}, F. Mignot and J.P. Puel \cite{MignotPuel}, we
define\begin{equation}\label{kappa}\kappa(\psi,u_0)= ess \inf
\{u\in\mathcal{P};\ u\geq\psi\ a.e.,\ u(0)\geq
u_0\}.\end{equation}This lower bound exists and is an element in
$\mathcal{P}$. Moreover, when $\psi$ is quasi-continuous, this
potential is the solution of the following reflected
problem:\begin{eqnarray*}\kappa\in\mathcal{P},\ \kappa\geq\psi,\
\frac{\partial\kappa}{\partial t}+A\kappa=0\ on\ \{u>\psi \},\
\kappa(0)=u_0.\end{eqnarray*} Mignot and Puel have proved in
\cite{MignotPuel} that $\kappa(\psi,u_0)$ is the limit (increasingly
and weakly in $L^2([0,T];\mathfrak{F})$)
of the solutions for the following penalized equations as  $\epsilon$ tends to $0$,
\begin{eqnarray*}u_\epsilon\in\mathcal{W},\ u_\epsilon(0)=u_0,\ \frac{\partial u_\epsilon}{\partial t}+
Au_\epsilon-\frac{(u_\epsilon-\psi)^-}{\epsilon}=0.\end{eqnarray*}
Let us point out that they obtain this result in the more general
case  where $\psi$ is only measurable from $[0,T[$ into
$L^2(\mathcal{O})$.\\We end this section by a convergence lemma which plays an important role in our approach (Lemma 3.8 in \cite{PIERRE}):
\begin{lemma}\label{convergemeas}
If $v^n\in\mathcal{P}$ is a bounded sequence in $\mathcal{K}$ and
converges weakly to $v$ in $L^2([0,T];\mathfrak{F})$; if $u$ is
a quasi-continuous function and $|u|$ is bounded by an element in
$\mathcal{P}$, then
$$\lim_{n\rightarrow+\infty}\int_0^T\int_\mathcal{O}ud\nu^{v^n}=\int_0^T\int_\mathcal{O}ud\nu^{v}.$$
\end{lemma}
\begin{remark}For the more general case one can see  Lemma 3.8 in \cite{PIERRE}. \end{remark}

\section{Quasi-continuitiy of the Solution of SPDE without Obstacle}\label{quasi-contSPDE}
Under assumptions {\bf(H)} and {\bf(I)}, SPDE \eqref{SPDE} with null Dirichlet boundary condition admits a unique solution which is denoted  by $\mathcal{U}(\xi,f,g,h)$. For this result, one can refer to \cite{Qiu}, in which the existence and uniqueness is given in the linear case. Then,  the nonlinear case can be solved  by Picard iteration.

\vspace{2mm}
The following is the main  result of this section.
\begin{theorem}\label{mainquasicontinuity} Under
assumptions {\bf(H)} and {\bf (I)}, $u=\mathcal{U}(\xi,f,g,h)$ the
solution of SPDE (\ref{SPDE})  admits a quasi-continuous version
denoted by $\tilde{u}$ i.e.
 $u=\tilde{u}$ $dP\otimes dt\otimes dx -$a.e. and for almost all $w\in\Omega$,
 $(t,x)\rightarrow \tilde{u}_t
(w,x)$ is quasi-continuous.
\end{theorem}
Before giving the proof of this theorem, we need the following lemmas. The first one is proved by Lemma 3.3 in \cite{PIERRE}:
\begin{lemma}\label{cap}There exists $C>0$ such that, for all open set $\vartheta\subset [0,T[\times \mathcal{O}$ and  $v\in\mathcal{P}$ with $v\geq1\ a.e.$ on $\vartheta$: $$cap\vartheta\leq C\| v\|^2_{\mathcal{K}}.$$ \end{lemma}
Let $\kappa=\kappa(u,u^+(0))$ be defined by relation \eqref{kappa}. One has to note that $\kappa$ is a random function. From now on, we always take for $\kappa$ the following measurable version
$$\kappa =\sup_n v^n,$$
where  $(v^n)$ is the non-decreasing sequence of random functions given by
\begin{equation} \label{eq:1}
\left\{ \begin{split}
         &\frac{\partial v_t^n}{\partial t}=Lv_t^n+n(v_t^n-u_t)^-\\
                  & v^n_0=u^+(0).
                          \end{split} \right.
                          \end{equation}
Using the results recalled in Subsection \ref{capacity}, we know that for almost all $w\in\Omega$, $v^n (w)$ converges weakly to $v(w)=\kappa(u(w),u^+(0)(w))$ in $L^2([0,T];\mathfrak{F})$ and that $v\geq u$.
\begin{lemma}\label{estimoftau} We have the following estimate:
\begin{equation}\label{estimateoftau}E\|\kappa\|_{\mathcal{K}}^2 \, \leq C\left(E\| u_0^+\|^2+E\| u_0\|^2+E\int_0^T\| f_t^0\|^2+\||g_t^0|\|^2+\||h_t^0|\|^2dt\right),\end{equation}
where $C$ is a constant depending only on the structure constants of the equation.
\end{lemma}
We can do a similar calculation as in the proof of Lemma 3 in \cite{DMZ12} to get \eqref{estimateoftau}. So we omit it and come to the proof of Theorem \ref{mainquasicontinuity}. 

\vspace{2mm}
\textbf{Proof of Theorem \ref{mainquasicontinuity}:} For simplicity, we set
$$
f_t(x)=f(t,x,u_t(x),\nabla u(x)), g_t(x)=g(t,x,u_t(x),\nabla u(x)), h_t(x)=h(t,x,u_t(x),\nabla u(x)).
$$
Let $(P_t)$ be the semi-group associated to operator $div(a\nabla)$, and for each positive integer $n$, we set
$$
u^n_0=P_{\frac{1}{n}}u_0\,,\quad  f^n_t=P_{\frac{1}{n}}f_t\,,\quad  g^n_t=P_{\frac{1}{n}}g_t \quad {\rm{and}}\quad h^n_{j,t}=P_{\frac{1}{n}}h_{j,t}\,.
$$
Fixing every positive integer $n$,  we define
\begin{eqnarray*}
u^n_t=P_tu^n_0+\int_{0}^{t}P_{t-s}f^n_sds+\int_{0}^{t}P_{t-s}div({\sigma^T g^n_s})ds+\sum_{j=1}^{\infty}\int_{0}^{t}P_{t-s}h^n_{j,s}dB^j_s\,.
\end{eqnarray*}
Firstly, we prove that the process $u^n$ is P-almost quasi-continuous.

Define an approximation sequence as follows, for every positive integer $m$,
\begin{eqnarray}\label{mild u^n,m}
u^{n,m}_t=P^m_tu^n_0+\int_{0}^{t}P^m_{t-s}f^n_sds+\int_{0}^{t}P^m_{t-s}div(\sigma^T g^n_s)ds+\sum_{j=1}^{\infty}\int_{0}^{t}P^m_{t-s}h^n_{j,s}dB^j_s,
\end{eqnarray}
where $(P^m_t)$ is the semigroup associated with the operator $A^m:=-div(a\nabla)-\frac{1}{m}\Delta$.

Let $(Q_t)$ be the semigroup on $L^2(\mathcal{O})$ associated with the generator $\Delta$, whose kernel is denoted by $q(t,x,y)$. Then it is known that the semigroup $P^m_t=Q_{\frac{t}{m}}P_t$. If the kernel of $P_t$ is denoted by $p(t,x,y)$, then by Fubini's theorem,
$$
P^m_t f(x)=\int_\mathcal{O} \int_\mathcal{O} q(\frac{t}{m},x,z) p(t,z,y)f(y)dydz=\int_{\mathcal{O}}\Big(\int_\mathcal{O}  q(\frac{t}{m},x,z) p(t,z,y)dz\Big)f(y)dy,
$$
it follows that the kernel $q_m(t,x,y)=\int_\mathcal{O}  q(\frac{t}{m},x,z) p(t,z,y)dz$ is smooth in $(t,x)$ (see $\cite{Aronson}$). Therefore, the process $u^{n,m}$ is P-almost surely continuous in $(t,x)$.

We consider a sequence of random open sets
$$
\theta_m=\{|u^{n,m+1}-u^{n,m}|>\epsilon\}.
$$
Let
$$
\kappa_m:=\kappa(\frac{1}{\epsilon}(u^{n,m+1}-u^{n,m}),0)+\kappa(-\frac{1}{\epsilon}(u^{n,m+1}-u^{n,m}),0),
$$
then by \cite{PIERRE}, there exists a constant $k>0$, such that for every positive integer $m$,

 $$cap(\theta_m)\leq \|\kappa_m\|^2_{\mathcal{K}}
 \leq k\frac{1}{\epsilon^2}\|u^{n,m+1}-u^{n,m}\|_{\mathcal{W}}^2 .$$


Since
\begin{eqnarray*}
du^{n,m}_t=div((a+\frac{1}{m}I)\nabla u^{n,m}_t)dt+f^n_t dt+div(\sigma g^n_t)dt+\sum_{j}^{\infty}h^n_{j,t}dB^j_t,
\end{eqnarray*}
it is well known that
\begin{eqnarray*}
&&E\sup_{t\in[0,T]}\|u^{n,m}_t\|^2+E\int_{0}^{T}\mathcal{E}(u^{n,m}_s,u^{n,m}_s)ds+\frac{1}{m}E\int_{0}^{T}\|\nabla u^{n,m}_s \|^2ds\\
&\leq& C E\left[\|u_0\|^2+\int_{0}^{T}\|f^n_s\|^2+\||g^n_s|\|^2+\||h^n_s|\|^2ds\right],
\end{eqnarray*}
where the upper bound on the right hand side is independent of $m$, which means the left hand side is uniformly bounded. Therefore, there exists a constant $C_n$, for every positive integer $m$,
\begin{eqnarray}
E\sup_{t\in[0,T]}\|u^{n,m}_t\|+E\int_{0}^{T}\mathcal{E}(u^{n,m}_s,u^{n,m}_s)ds+\frac{1}{m}E\int_{0}^{T}\|\nabla u^{n,m}_s \|^2ds\leq C_n\,.
\end{eqnarray}
For simplicity, we denote $\bar{u}=u^{n,m+1}-u^{n,m}$ in the following discussion. It is obvious that $\bar{u}$ satisfies the following equation:
$$
d\bar{u}_t=div(a\nabla \bar{u}_t)dt+\frac{1}{m+1}\Delta u^{n,m+1}_t dt-\frac{1}{m}\Delta u^{n,m}_t dt.
$$
Then
\begin{eqnarray*}
\|\bar{u}_t\|^2+2\int_{0}^{t}\mathcal{E}(\bar{u}_s,\bar{u}_s)ds&=&-2\int_{0}^{t}(\nabla \frac{u^{n,m+1}_s}{m+1}-\nabla \frac{u^{n,m}_s}{m},\nabla u^{n,m+1}_s-\nabla u^{n,m}_s)\\
&=&-\frac{2}{m+1}\int_{0}^{t}\|\nabla u^{n,m+1}_s \|^2ds-\frac{2}{m}E\int_{0}^{t}\|\nabla u^{n,m}_s \|^2ds\\
&+&2(\frac{1}{m}+\frac{1}{m+1})\int_{0}^{t}(\nabla u^{n,m}_s, \nabla u^{n,m+1}_s)ds\\
&\leq& -\frac{2}{m+1}\int_{0}^{t}\|\nabla u^{n,m+1}_s \|^2ds-\frac{2}{m}\int_{0}^{t}\|\nabla u^{n,m}_s \|^2ds\\
&+&\frac{2m+1}{m(m+1)}[\int_{0}^{t}\|\nabla u^{n,m+1}_s\|^2+\|\nabla u^{n,m}_s\|^2ds]\\
&\leq&  \frac{1}{m(m+1)}\int_{0}^{t}\|\nabla u^{n,m+1}_s \|^2ds.
\end{eqnarray*}
Therefore, $E\|u^{n,m+1}-u^{n,m}\|^2_{L^2(0,T;\mathfrak{F})}\rightarrow 0$ as $m\rightarrow \infty$.

By applying operator $\Delta$ on both sides of $(\ref{mild u^n,m})$, we get
\begin{eqnarray*}
\frac{1}{m}\Delta u^{n,m}_t&=&\frac{1}{m}\Big[\Delta Q_{\frac{t}{m}}(P_tu^n_0)+\int_{0}^{t}\Delta Q_{\frac{t-s}{m}}(P_{t-s}f^n_s)ds\nonumber\\
&&+\sum_i\int_{0}^{t}\Delta Q_{\frac{t-s}{m}}(P_{t-s}\partial_i \sigma_{ij}g^n_{s})ds+\sum_{j=1}^{\infty}\int_{0}^{t}\Delta Q_{\frac{t-s}{m}}(P_{t-s}h^n_{j,s})dB^j_s\Big]\,.\nonumber
\end{eqnarray*}

Since $\|\frac{d}{dt}Q_{\frac{t}{m}}\|\rightarrow 0$ as $m\rightarrow \infty$, it follows that $\lim_{m\rightarrow \infty}\|\frac{1}{m}\Delta u^{n,m}_t\|=0$.
 For
\begin{eqnarray*}
\frac{d(u^{n,m}_t-u^n_t)}{dt}=div(A\nabla (u^{n,m}_t-u^n_t))+\frac{1}{m}\Delta u^{n,m}_t,
\end{eqnarray*}
and by simple calculation, $\frac{d u^{n,m}_t}{dt}$ is convergent to
$\frac{du^n_t}{dt}$ in $L^2(0,T;\mathfrak{F}^\prime)$. This concludes that the approximation sequence $\{u^{n,m},m=1,2,\cdots\}$ is a Cauchy sequence in $\mathcal{W}$, so we pick a subsequence such that $E\|u^{n,m}-u^{n,m+1}\|^2_{\mathcal{W}}<\frac{1}{2^m}$.

Therefore, set $\epsilon=\frac{1}{m}$,
\begin{eqnarray*}
Ecap(\bigcup_{m=p}\theta_m)\leq \sum_{m=p}Ecap(\theta_m)\leq\sum_{m=p}\frac{1}{\epsilon^2}E\|u^{n,m}-n^{n,m+1}\|^2_{\mathcal{W}}
\leq (\sum_{m=p}\frac{m^2}{2^m})\rightarrow 0, \mbox{as}\ p\rightarrow \infty.
\end{eqnarray*}

For almost all $\omega\in \Omega$, $u^{n,m}$ is continuous and converges uniformly on $(\bigcup_{m=p}\theta_m)^c$, hence $u^n $ is continuous on  $(\bigcup_{m=p}\theta_m)^c$. This means $u^n $ is P-almost quasi-continuous.

\vspace{2mm}
Secondly, we  prove the $u$ as the limit of  $u^n$ is also quasi-continuous. It follows
 a similar argument  as in the proof of Theorem 3 in \cite{DMZ12}. \hfill$\Box$

\section{Existence and Uniqueness of Solution for the Obstacle Problem }

\textbf{Assumption (O):} The obstacle $S$ is assumed to be an adapted process, quasi-continuous, such that $S_0 \leq \xi$ $P$-almost surely and
controlled by the solution of a SPDE, i.e. $\forall t\in[0,T]$,
\begin{equation}
S_t\leq S'_t\nonumber
\end{equation} where
$S'$ is the solution of the linear SPDE
\begin{equation}
\left\{\begin{array}{ccl} \label{obstacle}
 dS'_t&=&LS'_tdt+f'_tdt+div(\sigma g'_t)dt+\sum_{j=1}^{+\infty}h'_{j,t}dB^j_t\\
                  S'(0)&=&S_0 ,
\end{array}\right. \end{equation}
where $S'_0\in L^2 (\Omega\times \mathcal{O})$ is
$\mathcal{F}_0\otimes\mathcal{B}(\mathcal{O})$-measurable, $f'$, $g'$ and $h'$ are adapted
processes respectively in $L^2
([0,T]\times\Omega\times\cO;\mathbb{R})$,  $L^2
([0,T]\times\Omega\times\cO;\mathbb{R}^d)$ and $L^2
([0,T]\times\Omega\times\cO;\mathbb{R}^{\mathbb{N}^*})$.
\begin{remark}We know that $S'$ uniquely exists and satisfies the following estimate (see \cite{Qiu}):
\begin{equation}\label{estimobstacle}E\sup_{t\in[0,T]}\| S'_t\|^2+E\int_0^T\|\sigma^T\nabla S'_t\|dt\leq CE\left[\| S'_0\|^2+\int_0^T(\| f'_t\|^2+\| |g'_t|\|^2+\| |h'_t|\|^2)dt\right]\end{equation}
Moreover, from Theorem \ref{mainquasicontinuity}, $S'$ admits a
quasi-continuous version.
\end{remark}
We now are able to define rigorously the notion of solution to the problem with obstacle we consider.
\begin{definition} A pair
$(u,\nu)$ is said to be a solution of the obstacle problem for
(\ref{SPDE}) if
\begin{enumerate}
    \item $u\in\mathcal{H}_T$ and $u(t,x)\geq S(t,x),\ dP\otimes dt\otimes
    dx-a.e.$ and $u_0(x)=\xi,\ dP\otimes dx-a.e.$;
    \item $\nu$ is a random regular measure defined on
    $[0,T)\times\mathcal{O}$;
    \item the following relation holds almost surely, for all
    $t\in[0,T]$ and $\forall\varphi\in\mathcal{D}$,
\begin{eqnarray}\label{solution}&&\hspace{-1cm}(u_t,\varphi_t)-(\xi,\varphi_0)-\int_0^t(u_s,\partial_s\varphi_s)ds+\int_0^t(\sigma^T \nabla u_s, \sigma^{T}\nabla \varphi_s)ds+\int_0^t(\sigma^T\nabla\varphi_s, g_s(u_s,\sigma^T\nabla u_s))ds\nonumber\\
&&\hspace{-1cm}=\int_0^t(\varphi_s, f_s(u_s,\sigma^T\nabla u_s))ds+\sum_{j=1}^{+\infty}\int_0^t(\varphi_s, h^j_s(u_s,\sigma^T\nabla u_s))dB^j_s+\int_0^t\int_{\mathcal{O}}\varphi_s(x)\nu(dx,ds)\,;
\end{eqnarray}
    \item $u$ admits a quasi-continuous version, $\tilde{u}$, and we have  $$\int_0^T\int_{\mathcal{O}}(\tilde{u}(s,x)-{S}(s,x))\nu(dx,ds)=0,\ \
    a.s..$$
  \end{enumerate}
\end{definition}
The main result of this paper is the following:
\begin{theorem}{\label{maintheo}}
Under assumptions {\bf (H)}, {\bf (I)} and {\bf (O)}, there exists a
unique weak solution of the obstacle problem for the SPDE
(\ref{SPDE}) associated to $(\xi,\ f,\ g,\ h,\ S)$.\\
We denote by $\mathcal{R}(\xi,f,g,h,S)$ the solution of SPDE
(\ref{SPDE}) with obstacle when it exists and is unique.
\end{theorem}
\noindent As the proof of this theorem is quite long, we split it in several steps. Firstly we prove existence and uniqueness in the linear case, then establish  It\^o's formula and finally prove the theorem by  fixed point argument.

\subsection{Proof of Theorem \ref{maintheo} in  the  linear case }\label{subsec5.2}
All along this subsection, we assume that $f$, $g$ and $h$ do not  depend on $u$ and
$\nabla u$, so we consider that $f$, $g$ and $h$  are adapted processes respectively in $L^2 ([0,T]\times\Omega\times\cO;\mathbb{R})$,  $L^2 ([0,T]\times\Omega\times\cO;\mathbb{R}^d)$ and $L^2 ([0,T]\times\Omega\times\cO;\mathbb{R}^{{\mathbb{N}^*}})$.\\
For $n\in \mathbb{N}^*$, let $u^n$ be the solution of the
following SPDE
\begin{equation}\label{penalization}
du_t^n=Lu_t^ndt+f_tdt+div(\sigma g_t)dt+\sum_{j=1}^{+\infty}h_{j,t}dB^j_t+n(u_t^n-S_t)^-dt
\end{equation}
with initial condition $u^n_0=\xi$ and null Dirichlet boundary
condition. We know from Theorem 2.2 in \cite{Qiu} that this equation
admits a unique solution in $\mathcal{H}_T$ and that the solution
admits $L^2(\mathcal{O})-$continuous trajectories.

\begin{lemma}\label{lemmaestim}
For all $n\in\bbN^*$, $u^n$ satisfies the following
estimate:\begin{eqnarray*}
E\sup_{t\in[0,T]}\| u_t^n\|^2+E\int_0^T\|\sigma^T\nabla u_t^n\|^2dt+E\int_0^Tn\|(u_t^n-S_t)^-\|^2dt\leq
C,\end{eqnarray*}
where $C$ is a constant depending only on the structure constants of the SPDE.\end{lemma}
\begin{proof} From
(\ref{penalization}) and (\ref{obstacle}), we know that $u^n-S'$
satisfies the following equation:
$$d(u_t^n-S'_t)=L(u_t^n-S'_t)dt+\tilde{f}_tdt+div(\sigma\tilde{g})_tdt+\sum_{j=1}^{+\infty}\tilde{h}^j_tdB^j_t+n(u_t^n-S_t)^-dt,$$
where $\tilde{f}=f-f'$, $\tilde{g}=g-g'$and $\tilde{h}=h-h'$.
Applying Ito's formula, we have
\begin{eqnarray*}
&&\|u^n_t-S'_t\|^2+2\int_{0}^{t}\|\sigma^T\nabla(u^n_s-S'_s)\|^2ds\\
&=&\|\xi-S'_0\|^2+2\int_{0}^{t}(u^n_s-S'_s,\tilde f_s)ds
-2\int_{0}^{t}(\sigma^T\nabla(u^n_s-S'_s),\tilde{g}_s)ds+\int_{0}^{t}\| |\tilde{h}_s|\|^2ds\\
&+&2\int_{0}^{t}(u^n_s-S'_s,n(u^n_s-S_s)^-)ds
+2\sum_{j=1}^{+\infty}\int_{0}^{t}(u^n_s-S'_s,\tilde{h}_s^j)dB^j_s\\
&\leq&\|\xi-S'_0\|^2+\epsilon\int_{0}^{t}\| u^n_s-S'_s\|^2ds+\frac{1}{\epsilon} \int_{0}^{t}\|\tilde f_s\|^2ds
+\epsilon\int_{0}^{t}\|\sigma^T\nabla(u^n_s-S'_s)\|^2ds\\&+&\frac{1}{\epsilon} \int_{0}^{t}\|\tilde{g}\|^2ds
+\int_{0}^{t}\| |\tilde{h}_s|\|^2ds
-2\int_{0}^{t} n\|(u^n_s-S_s)^-\|^2ds+2\sum_{j=1}^{+\infty}\int_{0}^{t}(u^n_s-S'_s,\tilde{h}^j_s)dB^j_s.
\end{eqnarray*}
By Burkholder-Davies-Gundy's inequality, it follows
\begin{eqnarray*}
E\sup_{t\in[0,T]}|\sum_{j=1}^{+\infty}\int_{0}^{t}(u^n_s-S'_s,\tilde{h}^j_s)dB^j_s|
&\leq& CE[\int_{0}^{T}\sum_{j=1}^{+\infty}|(u^n_s-S'_s,\tilde{h}^j_s)|^2ds]^{\frac{1}{2}}\nonumber\\
&\leq& \epsilon E\sup_{t\in[0,T]}\|u^n_s-S'_s\|^2+\frac{C}{4\epsilon}E[\int_{0}^{T}\| |\tilde{h}_s|\|^2ds].
\end{eqnarray*}
Then
\begin{eqnarray*}
\label{UBLC}
&&\hspace{-0.5cm}(1-\epsilon-\epsilon T) E\sup_{t\in[0,T]}\|u^n_s-S'_s\|^2+(2-\epsilon)E\int_{0}^{T}\|\sigma^T\nabla(u^n_s-S'_s)\|^2ds
+2E\int_0^Tn\|(u_s^n-S_s)^-\|^2ds\nonumber\\
&&\hspace{-0.5cm}\leq\, E\|\xi-S_0\|^2_{L^2}+\frac{1}{\epsilon}E\int_{0}^{T}\|\tilde f_s\|^2+\|\tilde{g}_s\|^2ds+(1+\frac{C}{4\epsilon})E\int_{0}^{t}\|\tilde{h}_s\|^2ds.
\end{eqnarray*}
We take $\epsilon$ small enough such that $(1-\epsilon-\epsilon T)>$ and $(2-\epsilon)>0$,
so that
\begin{equation*}
E\sup_{t\in[0,T]}\|u_t^n-S'_t\|^2+E\int_0^T\|\sigma^T\nabla(u_t^n-S'_t)\|^2dt+E\int_0^Tn\|(u_t^n-S_t)^-\|^2dt\leq C.
\end{equation*}
Hence, by \eqref{estimobstacle}, we obtain the desired estimate.
\end{proof}

Furthemore, we can do a similar argument as in \cite{DMZ12} to end the proof of Theorem \ref{maintheo}.

\begin{remark} \label{remark.quasi-continuity}
	We have to point out that $u$ lies not only in $\mathcal{H}_T$, which can be shown by last lemma, as a limit of a subsequence of $\{u^n\}$.  By It\^o's formula proved in the next subsection,  we find $u\in \mathcal{K}$ a.s.. The quasi-continuity of the solution $u$ is obtained since
	it can be decomposed by two parts, $u=v+z$, where  
	$$
	dv_t=Lv_tdt+n(u_t^n-S_t)^-dt\ \ \mbox{and}\ \
	dz_t=Lz_tdt+f_tdt+div(\sigma g_t)dt+\sum_{j=1}^{+\infty}h_{j,t}dB^j_t.
	$$
	By \cite{DenisStoica},  it is known that $z\in \mathcal{K}$, a.s. and $z$ allows a quasi-continuous version $\tilde z$.  Therefore,  $\nu=\partial_t v+A v$ is a regular measure and $v$ has a quasi-continuous version $\tilde v$ (see \cite{MignotPuel}, \cite{PIERRE} ), which implies that $\tilde u=\tilde v+\tilde z$ is a quasi-continuous version of $u$.
	
\end{remark}

\subsection{It\^o's formula}
The following It\^o formula for the solution of the obstacle problem is fundamental to get all the results in the nonlinear case. Let us also remark that any solution of the nonlinear equation \eqref{SPDE} may be viewed as the solution of a linear one so that it satisfies also It\^o's formula.
\begin{theorem}\label{Itoformula}
Under assumptions of the previous subsection \ref{subsec5.2}, let $u$ be the solution of SPDE(\ref{SPDE}) with obstacle and
$\Phi:\mathbb{R}^+\times\mathbb{R}\rightarrow\mathbb{R}$ be a
function of class $\mathcal{C}^{1,2}$. We denote by $\Phi'$ and
$\Phi''$ the derivatives of $\Phi$ with respect to the space
variables and by $\frac{\partial\Phi}{\partial t}$ the partial
derivative with respect to time. We assume that there exsits a constant $C>0$, such that $|\Phi''|\leq C$, $|\frac{\partial \Phi}{\partial t}|\leq C(|x|^2 \vee 1)$,  and $\Phi'(t,0)=0$ for all $t\geq0$. Then, $P-a.s.$ for
all $t\in[0,T]$,
\begin{eqnarray*}
&&\hspace{-0.8cm}\int_\mathcal{O}\Phi(t,u_t(x))dx+\int_0^t(\sigma^T\nabla\Phi'(s,u_s),\sigma^T\nabla u_s)ds=\int_\mathcal{O}\Phi(0,\xi(x))dx+\int_0^t\int_\mathcal{O}\frac{\partial\Phi}{\partial s}(s,u_s(x))dxds\\&&\hspace{-0.8cm}+\int_0^t(\Phi'(s,u_s),f_s)ds
-\int_\mathcal{O}(\Phi''(s,u_s)\sigma^T\nabla u_s, g_s)ds+\sum_{j=1}^{+\infty}\int_0^t(\Phi'(s,u_s),h_j)dB_s^j\\&&\hspace{-0.8cm}
+\frac{1}{2}\sum_{j=1}^{+\infty}\int_0^t\int_\mathcal{O}\Phi''(s,u_s(x))(h_{j,s}(x))^2dxds+\int_0^t\int_\mathcal{O}\Phi'(s,\tilde{u}_s(x))\nu(dxds).
\end{eqnarray*}
\end{theorem}
We still consider $(u ,\nu)$ the solution of the linear equation as in Subsection \ref{subsec5.2}
\begin{equation*}\left\{\begin{array}{ccl}
du_t&=&div(a\nabla u_t)dt+f_tdt+ div(\sigma g_t)dt+\sum_{j=1}^{\infty}h^j_tdB^j_t+\nu(x,dt)\\
u&\geq&S
\end{array}\right. \end{equation*}
and consider another linear equation with adapted coefficients $\bar{f}$, $\bar{g}$, $\bar{h}$ respectively in $L^2 ([0,T]\times\Omega\times\cO;\mathbb{R})$,  $L^2 ([0,T]\times\Omega\times\cO;\mathbb{R}^d)$ and $L^2 ([0,T]\times\Omega\times\cO;\mathbb{R}^{{\mathbb{N}^*}})$ and obstacle $\bar{S}$ which satisfies the same hypotheses ${\bf (O)}$ as $S$ i.e; $\bar{S}_0 \leq \xi$ and $\bar{S}$ is dominated by the solution of an SPDE (not necessarily the same as $S$). We denote by $(y, \bar{\nu})$ the unique solution to the associated SPDE with obstacle with initial condition $y_0 =u_0 =\xi$.
\begin{equation*}\left\{\begin{array}{ccl}
dy_t&=&div(a\nabla y_t)dt+\bar{f}_tdt+ div(\sigma\bar{g}_t)dt+\sum_{j=1}^{\infty}\bar{h}^j_tdB^j_t+\bar{\nu}(x,dt)\\
y&\geq&\bar{S}
\end{array}\right. \end{equation*}

\begin{theorem}\label{Itodifference}Let $\Phi$ as in Theorem \ref{Itoformula}, then the difference of the two solutions satisfies the following It\^o
formula for all $t\in [0,T]$:\begin{eqnarray}\label{itodifference}&&\int_\mathcal{O}\Phi(t,u_t(x)-y_t(x))dx+\int_0^t(\sigma^T\nabla\Phi'(s,u_s-y_s),\sigma^T\nabla(u_s-y_s))ds\nonumber\\&&=\int_0^t(\Phi'(s,u_s-y_s),f_s-\bar{f}_s)ds
-\int_0^t(\Phi''(s,u_s)\sigma^T\nabla u_s,g_s-\bar{g}_s)ds\nonumber\\&&+\sum_{j=1}^{+\infty}\int_0^t(\Phi'(s,u_s-y_s),h^j_s-\bar{h}^j_s)dB^j_s+\frac{1}{2}\sum_{j=1}^{+\infty}\int_0^t\int_\mathcal{O}\Phi''(s,u_s-y_s)
(h^j_s-\bar{h}^j_s)^2dxds\nonumber\\&&+\int_0^t\int_\mathcal{O}\frac{\partial\Phi}{\partial
s}(s,u_s-y_s)dxds+\int_0^t\int_\mathcal{O}\Phi'(s,\tilde{u}_s-\tilde{y}_s)(\nu-\bar{\nu})(dx, ds),\quad a.s..
\end{eqnarray}\end{theorem}
The proof the above theorems is similar to that of Theorem 5 and Theorem 6 in \cite{DMZ12}.

\subsection{Proof of Theorem 2 in  the nonlinear case}

Let us consider the Picard sequence $(u^n)_n$
defined by $u^0=\xi$ and for all $n\in\mathbb{N}^*$ we denote by $(u^{n+1},\nu^{n+1})$ the solution of the linear OSPDE 
$$(u^{n+1}, \nu^{n+1})=\mathcal{R} (\xi, f(u^n,\sigma^T\nabla u^n), g(u^n,\sigma^T\nabla u^n), h(u^n,\sigma^T\nabla u^n), S).$$
Then, by It\^o's formula (\ref{itodifference}), we have almost surely
\begin{eqnarray*}&&e^{-\gamma T}\|
u_T^{n+1}-u_T^n\|^2+\int_0^Te^{-\gamma
s}\|\sigma^T\nabla(u_s^{n+1}-u_s^n)\|^2ds=-\gamma\int_0^Te^{-\gamma
s}\| u_s^{n+1}-u_s^n\|^2ds\\&&+2\int_0^Te^{-\gamma
s}(\hat{f}_s,u_s^{n+1}-u_s^n)ds-2\int_0^Te^{-\gamma
s}(\hat{g}_s,\sigma^T\nabla(u_s^{n+1}-u_s^n))ds+\int_0^Te^{-\gamma
s}\||\hat{h}_s|\|^2ds\\&&+2\sum_{j=1}^{+\infty}\int_0^Te^{-\gamma
s}(\hat{h}^j_s,u_s^{n+1}-u_s^n)dB^j_s+2\int_0^T\int_\mathcal{O}e^{-\gamma
s}(u_s^{n+1}-u_s^n)(\nu^{n+1}-\nu^n)(dxds)\,,
\end{eqnarray*}
 where
$\hat{f}=f(u^n,\sigma^T\nabla u^n)-f(u^{n-1},\sigma^T\nabla u^{n-1})$,
$\hat{g}=g(u^n,\sigma^T\nabla u^n)-g(u^{n-1},\sigma^T\nabla u^{n-1})$ and
$\hat{h}=h(u^n,\sigma^T\nabla u^n)-h(u^{n-1},\sigma^T\nabla u^{n-1})$.
 Clearly, the last term is non-positive, so using Cauchy-Schwarz's inequality and the Lipschitz conditions on $f$, $g$ and $h$, we have
\begin{eqnarray*}
2\int_0^Te^{-\gamma s}(u_s^{n+1}-u_s^n,\hat{f}_s)ds
&\leq&\frac{1}{\epsilon}\int_0^Te^{-\gamma s}\| u_s^{n+1}-u_s^n\|^2ds+\epsilon\int_0^Te^{-\gamma s}\|\hat{f}_s\|^2ds
\\&\leq&\frac{1}{\epsilon}\int_0^Te^{-\gamma s}\| u_s^{n+1}-u_s^n\|^2ds+C\epsilon\int_0^Te^{-\gamma s}\| u_s^n-u_s^{n-1}\|^2ds\\&+&C\epsilon\int_0^Te^{-\gamma s}
\|\sigma^T\nabla(u_s^n-u_s^{n-1})\|^2ds\end{eqnarray*}
and
\begin{eqnarray*}&&2\int_0^Te^{-\gamma
s}(\hat{g}_s,\sigma^T\nabla(u_s^{n+1}-u_s^n))ds\\&\leq&2\int_0^Te^{-\gamma
s}\|\sigma^T\nabla(u_s^{n+1}-u_s^n)\|\Big(C\|
u_s^n-u_s^{n-1}\|+\alpha\|\sigma^T\nabla(u_s^n-u_s^{n-1})\|\Big)ds\\&\leq&
C\epsilon\int_0^Te^{-\gamma
s}\|\sigma^T\nabla(u_s^{n+1}-u_s^n)\|^2ds+\frac{C}{\epsilon}\int_0^Te^{-\gamma
s}\|u_s^n-u_s^{n-1}\|^2ds\\&+&\alpha\int_0^Te^{-\gamma
s}\|\sigma^T\nabla(u_s^{n+1}-u_s^n)\|^2ds+\alpha\int_0^Te^{-\gamma
s}\| \sigma^T\nabla(u_s^n-u_s^{n-1})\|^2ds\end{eqnarray*}
and
\begin{eqnarray*}
\int_0^Te^{-\gamma s}\||\hat{h}_s|\|^2ds\leq C^2(1+\frac{1}{\epsilon})\int_0^Te^{-\gamma s}\| u_s^n-u_s^{n-1}\|^2ds
+\beta^2(1+\epsilon)\int_0^Te^{-\gamma
s}\|\sigma^T\nabla(u_s^n-u_s^{n-1})\|^2ds\end{eqnarray*} where
$C$, $\alpha$ and $\beta$ are the constants in the Lipschitz
conditions. Taking expectation, we
get:
\begin{equation*}
\begin{split}
&(\gamma-\frac{1}{\epsilon})E\int_0^Te^{-\gamma s}\|u_s^{n+1}-u_s^n\|^2ds
+(1-\alpha-C\epsilon)E\int_0^Te^{-\gamma s}\|\sigma^T\nabla(u_s^{n+1}-u_s^n)\|^2ds\\&\hspace{-0.1cm}
\leq C(C+\epsilon+\frac{C+1}{\epsilon})\int_0^Te^{-\gamma s}\|u_s^n-u_s^{n-1}\|^2ds
+(C\epsilon+\alpha+\beta^2(1+\epsilon))E\int_0^Te^{-\gamma s}\|\sigma^T\nabla(u_s^n-u_s^{n-1})\|^2ds\,.\end{split}\end{equation*}
We choose $\epsilon$ small enough and then $\gamma$ such that
$$C\epsilon+\alpha+\beta^2(1+\epsilon)<1-\alpha-C\epsilon\ \ \rm{and}\ \ \frac{\gamma-1/\epsilon}{1-\alpha-C\epsilon}=\frac{C(C+\epsilon+(C+1)/\epsilon)}{C\epsilon+\alpha+\beta^2(1+\epsilon)}$$
Set $\delta=\frac{\gamma-1/\epsilon}{1-\alpha-C\epsilon}$, we define the norm on $L^2(\Omega\times[0,T];\mathfrak{F})$,
$$
\|u\|_{\gamma, \delta}=E[\int_0^T e^{-\gamma s}(\delta\|u_s\|^2+\|\sigma^T \nabla u_s\|^2)ds],
$$
which is an equivalent norm of $L^2(\Omega\times[0,T];\mathfrak{F})$.

We have the following inequality: $$\|
u^{n+1}-u^n\|_{\gamma,\delta}\leq\frac{C\epsilon+\alpha+\beta^2(1+\epsilon)}{1-\alpha-C\epsilon}\|
u^n-u^{n-1}\|_{\gamma,\delta}\leq...\leq(\frac{C\epsilon+\alpha+\beta^2(1+\epsilon)}{1-\alpha-C\epsilon})^n\|
u^1\|_{\gamma,\delta}$$ when $n\rightarrow\infty$,
$(\frac{C\epsilon+\alpha+\beta^2(1+\epsilon)}{1-\alpha-C\epsilon})^n\rightarrow0$,
we deduce that $(u^n)_n$ converges strongly to $u$ in
$L^2(\Omega\times[0,T];\mathfrak{F})$. \\Moreover, as
$(u^{n+1},\nu^{n+1})=\mathcal{R} (\xi, f(u^n,\sigma^T\nabla u^n), g(u^n,\sigma^T\nabla u^n), h(u^n,\sigma^T\nabla u^n), S)$, we have for any $\varphi\in\mathcal{D}$:
 \begin{eqnarray*}\label{weaksol}&&\hspace{-1cm}(u^{n+1}_t,\varphi_t)-(\xi,\varphi_0)-\int_0^t(u^n_s,\partial_s\varphi_s)ds+\int_0^t(\sigma^T\nabla u^{n+1}_s,\sigma^T\nabla\varphi_s)ds+\int_0^t(g_s(u_s^n,\sigma^T\nabla u_s^n),\sigma^T\nabla\varphi_s)ds\nonumber\\
&=&\int_0^t(f_s(u_s^n,\sigma^T\nabla u_s^n),\varphi_s)ds+\sum_{j=1}^{+\infty}\int_0^t(h^j_s(u_s^n,\sigma^T\nabla u_s^n),\varphi_s)dB^j_s+\int_0^t\int_{\mathcal{O}}\varphi_s(x)\nu^{n+1}(dxds),\ \ a.s..\end{eqnarray*}
Let $v^{n+1}$ the random parabolic potential associated to $\nu^{n+1} $:
$$\nu^{n+1}=\partial_t v^{n+1}+Av^{n+1}.$$ We denote
$z^{n+1}=u^{n+1}-v^{n+1}$, so
$$z^{n+1} =\mathcal{U}(\xi,f(u^n,\sigma^T\nabla u^n), g(u^n,\sigma^T\nabla u^n), h(u^n,\sigma^T\nabla u^n))$$ converges strongly to $z$
in $L^2([0,T];\mathfrak{F})$. As a consequence of the strong
convergence of $(u^{n+1})_n$, we deduce that $(v^{n+1})_n$ converges
strongly to $v$ in $L^2([0,T];\mathfrak{F})$. Therefore, for fixed $\omega$,
\begin{equation*}\begin{split}\int_0^t(-\frac{\partial_s\varphi_s}{\partial s},v_s)ds&+\int_0^t(\sigma^T\nabla\varphi_s,\sigma^T\nabla v_s)ds\\&=\lim\limits_{n\rightarrow\infty}\int_0^t(-\frac{\partial_s\varphi_s}{\partial s},v^{n+1}_s)ds+\int_0^t(\sigma^T\nabla\varphi_s,\sigma^T\nabla v^{n+1}_s)ds\geq0,\end{split}
\end{equation*}
i.e. $v(\omega)\in\mathcal{P}$. Then from Proposition \ref{presentation},
we obtain a regular measure associated with $v$, and $(\nu^{n+1})_n$
converges vaguely to $\nu$. \\Taking the limit, we obtain
\begin{eqnarray*}&&\hspace{-0.6cm}(u_t,\varphi_t)-(\xi,\varphi_0)-\int_0^t(u_s,\partial_s\varphi_s)ds+\int_0^t(\sigma^T\nabla u_s,\sigma^T\nabla\varphi_s)ds+\int_0^t(\sigma^T\nabla\varphi_s, g_s(u_s,\sigma^T\nabla u_s))ds\nonumber\\
    &&\hspace{-0.6cm}=\int_0^t(f_s(u_s,\sigma^T\nabla u_s),\varphi_s)ds+\sum_{j=1}^{+\infty}\int_0^t(h^j_s(u_s,\sigma^T\nabla u_s),\varphi_s)dB^j_s+\int_0^t\int_{\mathcal{O}}\varphi_s(x)\nu(dx,ds),\ \ a.s..\end{eqnarray*}
From the fact that $u$ and $z$ are in $\mathcal{H}_T$, we know that
$v$ is also in $\mathcal{H}_T$, by definition, $\nu$ is a random
regular measure.         $ \hfill \Box$

\subsection{Comparison theorem}

We consider $(u,\nu )=\mathcal{R} (\xi, f,g,h,S)$ the solution of the SPDE with obstacle
\begin{equation}\left\{ \begin{split}&du_t(x)=Lu_t(x)dt+f(t,x,u_t(x),\sigma^T\nabla
u_t(x))dt+div\sigma g(t,x,u_t(x),\sigma^T\nabla
u_t(x))dt\nonumber\\&\ \ \ \ \quad \quad+\sum_{j=1}^{+\infty}h_j(t,x,u_t(x),\sigma^T\nabla
u_t(x))dB^j_t +\nu (x, dt)\\
&u\geq S\ ,\  u_0=\xi ,\  \end{split}\right.\end{equation}
where we assume hypotheses {\bf (H)}, {\bf (I)} and {\bf (O)}. \\
We consider another coefficients $f'$ which satisfies the same assumptions as $f$, another obstacle $S'$ which satisfies {\bf (O)} and another initial condition $\xi'$ belonging to $L^2 (\Omega \times \cO)$ and $\mathcal{F}_0$ adapted such that $\xi'\geq S'_0$. We denote by $(u' ,\nu' )=\mathcal{R} (\xi', f',g,h,S')$.
\begin{theorem} \label{comparison}Assume that the following conditions
hold\begin{enumerate}
      \item $\xi\leq\xi',\ dx\otimes dP-a.e.$
      \item $f(u,\nabla u)\leq f'(u,\nabla u),\ dt\otimes dx\otimes dP-a.e.$
      \item $S\leq S',\ dt\otimes dx\otimes dP-a.e.$
    \end{enumerate}
Then for almost all $\omega\in\Omega$, $u(t,x)\leq u'(t,x),\
q.e..$\end{theorem}
As we have It\^o's formula for the difference between the solutions of two OSPDEs, the proof of the comparison theorem is more or less classic. One can refer to \cite{DMZ12}.


\section{$L^p$-estimates of the weak solution}
In this section, we  assume further the H\"{o}rmander condition is satisfied.  By the method of De Giorgi iteration, we obtain the $L^p$-estimates on the weak solutions for the  linear OSPDEs with higher integrability of the parameters. Then the maximum principle can be obtained. For simplicity, we consider the linear case, i.e. the coefficients $f,g$ and $h$ do not depend on $u$ and $\nabla u$.

\vspace{2mm}
Sobolev spaces with fractional orders will be use in this section. For $s\in \mathbb{R}$, we denote by $H^s(\mathbb{R}^d):=(I-\Delta)^{-\frac{s}{2}}L^2(\mathbb{R}^d)$ the Bessel potential space with the norm
$$
\|f\|_{H^s(\mathbb{R}^d)}:=\|(I-\Delta)^{\frac{s}{2}}f\|_{L^2},\ \  f\in H^s.
$$
Let $H^s(\mathcal{O})$ be the restriction of $H^s(\mathbb{R}^d)$  on $\mathcal{O}\subset \mathbb{R}^d$, and for $f\in H^s(\mathcal{O})$, define
$$
\|f\|_{H^s(\mathcal{O})}=\inf\{\|g\|_{H^s(\mathbb{R}^d)}: g\in H^s(\mathbb{R}^d), g|_{\mathcal{O}}=f\}.
$$
If $s$ is an integer, the Bessel potential space coincides (with equivalence of norms)  with  the Sobolev space with integer order.

Since Sobolev embedding inequality will be used in this section, we assume the domain $\mathcal{O}$ has smooth boundary.

Define $BC^{\infty}_{b}$ as the set of real-valued measurable function $f$ on $\Omega\times[0,T]\times\mathcal{O}$, such that for each $\omega\in \Omega, t\in[0,T]$, the function $f(\omega,t,x)$ is infinite differentiable with respect to $x$, and all the derivatives of any order belongs to $L^{\infty}([0,T]\times\mathcal{O})$.

Recall the diffusion matrix $a=(a_{ij})$ where $a_{ij}=\sum_{k=1}^{d}\sigma_{ik}\sigma_{jk}$. We assume the first order differential operators
$$
L_k=\sum_{i=1}^{d} \sigma_{ik}\partial_i,\ \  k=1,...,n.
$$
Set
$$
\mathbb{L}_0=\{L_1,...,L_n\}\quad \mbox{and} \quad \mathbb{L}_{n+1}=\mathbb{L}_{n}\cup\{[L_k, M]: M\in \mathbb{L}_{n},  k=1,...,n\},
$$
where $[L_k, M]=L_kM-ML_k$. Denote by $Lie_n$ the set of linear combinations of elements of $\mathbb{L}_n$ with coefficients of $BC^{\infty}_b$.

\vspace{3mm}
We assume that the following H\"{o}rmander-type condition is satisfied in this section.

\vspace{3mm}
\textbf{(HA)} There exists a non-negative integer $n_0 $ such that $\{\frac{\partial}{\partial x_i},i=1,\cdots,d\}\in Lie_{n_0}$, and we assume,  if $d=1$,  $n_0\geq 2$; if $d=2$, $n_0\geq 1$; if $d\geq 3$, $n_0\geq 0$.

\vspace{2mm}
The following lemma is a stochastic version of Lemma 4.2 in \cite{Krylov}.
\begin{lemma}
\label{hormander lemma}
 For $\{L,K\}\subset\cup_{l\geq 0} \mathbb{L}_l$ and $\epsilon\in[0,1]$, there exists a constant $C>0$ such that almost surely for any $f \in H^0$ with $Lf\in H^{-1+\epsilon}$ and $Kf\in H^0$, it holds that
$$
\|[L,K]f\|_{H^{-1+\frac{\epsilon}{2}}}\leq C(\|Lf\|_{H^{-1+\epsilon}}+\|Kf\|_{H^{0}}+\|f\|_{H^0}).
$$
\end{lemma}

\vspace{3mm}

Keeping the above lemma in mind, we recall the uniform boundedness estimates in the proof of Lemma \ref{lemmaestim} in Section 5:
\begin{eqnarray}
\label{estimates on u-S}
E\sup_{t\in[0,T]}\|u_s-S'_s\|^2&+&E\int_{0}^{T}\sum_{k=1}^n\|L_k (u_s-S'_s)\|^2ds\nonumber\\
&\leq& C\left\{ E\|\xi-S_0\|^2+E\int_{0}^{T}\|\tilde f_s\|^2
+\|\tilde{g}_s\|^2+\|\tilde{h}_s\|^2ds\right\},
\end{eqnarray}
where $\tilde{f}=f-f'$, $\tilde{g}=g-g'$ and $\tilde{h}=h-h'$.
\vspace{2mm}

Since $\|\cdot\|_{H^{s}}\leq \|\cdot\|_{H^{0}}$, for $s\leq0$, applying  the Lemma $\ref{hormander lemma}$ and above estimates repeatedly, by the assumption \textbf{(HA)}, we obtain the following result:

\begin{theorem}
Under assumption \textbf{(H)}, \textbf{(I)}, \textbf{(O)} and \textbf{(HA)}, let $\eta=\frac{1}{2^{n_0}}$, then $u_t-S'_t\in L^2(\Omega\times[0,T]; H^\eta)$ a.s. and
\begin{equation*}
E\int_{0}^{T}\|u_s-S'_s\|^2_{H^{\eta}}\leq N\left\{ E\|\xi-S_0\|^2+E\int_{0}^{T}\|\tilde f_s\|^2
+\|\tilde{g}_s\|^2+\|\tilde{h}_s\|^2ds\right\},
\end{equation*}
where the positive constant $N$  depends on $T$, $n_0$ and the structure constants of the OSPDE.
\end{theorem}

The iteration sequence is constructed as follows. For $\lambda>0$ and $m\in \mathbb{N}_0$, set
\begin{eqnarray*}
&&v=u-S',\ \ v^m=[v-\lambda(1-\frac{1}{2^m})]^+\,, \\
&&V^m_t=\sup_{t\in[0,T]}\|v^m_t\|^2+\int_{0}^{T}\sum_{k=1}^n\|L_k v^m_s\|^2+\|v^m_s\|^2_{H^{\eta}}ds\,.
\end{eqnarray*}

The following properties of  sequence $\{v^m\}$ will be used a lot in this section
\vspace{2mm}.

For every $m\in \mathbb{N}$,  $v^m\leq v^{m-1}$, $vI_{\{v^m>0\}}=v^m+\lambda(1-\frac{1}{2^m})I_{\{v^m>0\}}$,
$|\frac{\partial v^m}{\partial x_i}|\leq |\frac{\partial v^{m-1}}{\partial x_i}|$, for $i=1,\cdots,n$, and
$I_{\{v^m>0\}}\leq \left(\frac{2^m v^{m-1}}{\lambda}\right)^{q}$, for $\forall q>0$.

\vspace{3mm}
Since the process $v_t$ satisfies the following OSPDE:
\begin{eqnarray}
dv_t=(div(A\nabla v_t)+ \tilde{f}_t+div(\sigma \tilde{g}_t))dt+\sum_{j} \tilde{h}^{j}_t dB^j_t+\nu(x,dt)\,,\nonumber
\end{eqnarray}
by Ito's formula, we obtain
\begin{eqnarray}
&&\|v^m_t\|^2+2\int_{0}^{t}\sum_{k=1}^{n}\|L_k v^m\|^2ds\nonumber\\
&=& \|v^m_0\|^2+2\int_{0}^{t}(\tilde{f}_s, v^m_s)ds-2\int_{0}^{t}(\tilde{g}_s, \sigma ^T\nabla v^m_s)ds
+2\sum_{j}\int_{0}^{t}(v^m_s, \tilde{h}^{j}_s)dB^j_s\nonumber\\
&&+\int_{0}^{t}\|I_{\{v^m_s>0\}}\tilde{h}_s\|^2ds+2\int_{0}^{t}\int_{\mathcal{O}} v^m_s \nu(dxds) \nonumber\\
&\leq&  \|v^m_0\|^2+2\int_{0}^{t}(\tilde{f}_s, v^m_s)ds
+\epsilon \int_{0}^{t}\sum_{k=1}^{n}\|L_k v^m_s\|^2ds+\frac{1}{\epsilon}\int_{0}^{t}\|\tilde{g}_s I_{\{v^m_s>0\}}\|^2ds\nonumber\\
&&+2\sum_{j}\int_{0}^{t}(v^m_s, \tilde{h}^{j}_s)dB^j_s+\int_{0}^{t}\|I_{\{v^m_s>0\}}\tilde{h}_s\|^2ds\,,\nonumber
\end{eqnarray}
where $\epsilon>0$. The support of $\nu$ is $\{v=0\}$ and $v=0$ implies $v^m=0$, hence $\int_{0}^{t}\int_{\mathcal{O}} v^m_s \nu(dxds)=0$. Set $\epsilon=1$ and $M^m_t:=\sum_{j}\int_{0}^{t}(v^m_s, \tilde{h}^{j}_s)dB^j_s$. It follows that
\begin{eqnarray}
&&\|v^m_t\|^2+\int_{0}^{t}\sum_{k=1}^{n}\|L_k v^m\|^2ds\nonumber\\
&\leq& \|v^m_0\|^2+2\int_{0}^{t} (\tilde{f}_s, v^m_s)+\|\tilde{g}_s I_{\{v^m_s>0\}}\|^2+\|I_{\{v^m_s>0\}}\tilde{h}_s\|^2ds
+2\sup_{t\in[0,T]}|M^m_t|.\nonumber
\end{eqnarray}
Applying Lemma $\ref{hormander lemma}$ and above estimates repeatedly, by the assumption \textbf{(HA)}, we obtain that, there exists a constant $C>0$ depending on $T$ and $n_0$ such that
\begin{eqnarray}
\label{vm estimates}
&&\sup_{t\in[0,T]}\|v^m_t\|^2+\int_{0}^{T}\|v^m_s\|^2_{H^\eta}+\sum_{k=1}^{n}\|L_k v^m_s\|^2ds\\
&\leq& C\left(\|v^m_0\|^2+2\int_{0}^{T}(\tilde{f}_s, v^m_s)+\|\tilde{g}_s I_{\{v^m_s>0\}}\|^2+\|I_{\{v^m_s>0\}}\tilde{h}_s\|^2ds
+2\sup_{t\in[0,T]}|M^m_t|\right).\nonumber
\end{eqnarray}
Therefore, we have
\begin{eqnarray}
\label{estimate Vm}
V^m\leq C\left(\|v^m_0\|^2+\int_{0}^{T} (\tilde{f}_s, v^m_s)+\|\tilde{g}_s I_{\{v^m_s>0\}}\|^2+\|I_{\{v^m_s>0\}}\tilde{h}_s\|^2ds
+\sup_{t\in[0,T]}|M^m_t|\right).
\end{eqnarray}

\vspace{2mm}
Since $v^m\in L^{2}([0,T];H^{\eta})\cap C([0,T]; L^2)$ a.s., we are able to embed $v^m$ into the space with higher integrability in the following lemma.
\begin{lemma}
 There is a constant $C>0$ depending on $n_0$ and $d$ such that
\begin{eqnarray}
\label{estimate vm- q}
\|v^m\|_{L^{\frac{2(d+2\eta)}{d}}([0,T]\times\mathcal{O})}\leq  C\|v^m\|^{\frac{2\eta}{d+2\eta}}_{C([0,T]; L^2)}
\cdot \|v^m\|_{ L^{2}([0,T];H^{\eta})}^{\frac{d}{d+2\eta}},\ a.s.
\end{eqnarray}
and
\begin{eqnarray}
\label{vm<VM}
\|v^m\|^2_{L^{\frac{2(d+2\eta)}{d}}([0,T]\times\mathcal{O})}\leq CV^m,\ a.s..
\end{eqnarray}
\end{lemma}
\begin{proof}
By H\"{o}lder's inequality, setting $\frac{1}{l}+\frac{1}{r}=1$, $\alpha<1$ and constant $q>0$ which will be valued later, we obtain
\begin{eqnarray}
\|v^m_s\|^{q}_{L^q}&:=&\int_{\mathcal{O}}|v^m_s(x)|^q dx\nonumber\\
&=&\int_{\mathcal{O}}|v^m_s(x)|^{q\alpha} |v^m_s(x)|^{q(1-\alpha)}dx\nonumber\\
&\leq&\left( \int_{\mathcal{O}}|v^m_s(x)|^{q\alpha l} dx\right)^{\frac{1}{l}}\cdot \left( \int_{\mathcal{O}}|v^m_s(x)|^{q(1-\alpha) r} dx\right)^{\frac{1}{r}}\nonumber\\&=&\|v^m_s\|_{L^{q\alpha l}}^{q\alpha}\cdot \|v^m_s\|_{L^{q(1-\alpha)r}}^{q(1-\alpha)}.\nonumber
\end{eqnarray}
By Sobolev embedding inequality, set $q\alpha l=\frac{2d}{d-2\eta}$, then
\begin{eqnarray}
\|v^m_s\|_{L^{q\alpha l}}\leq C\|v^m_s\|_{H^{\eta}}.\nonumber
\end{eqnarray}
Therefore,
\begin{eqnarray}
\int_{0}^{T}\|v^m_s\|^{q}_{L^q}ds\leq C\sup_{t\in[0,T]} \|v^m_t\|_{L^{q(1-\alpha)r}}^{q(1-\alpha)}\cdot \int_{0}^{T}\|v^m_s\|_{H^{\eta}}^{q\alpha}ds.\nonumber
\end{eqnarray}

Letting $q\alpha=2, q(1-\alpha)r=2$,  we know the right hand side of above inequality is finite. By further simple calculation, we obtain
 $\alpha=\frac{d}{d+2\eta}$, $q=\frac{2(d+2\eta)}{d}$, $r=\frac{d}{2\eta}$ and $l=\frac{d}{d-2\eta}$. Hence,
\begin{eqnarray*}
\|v^m\|_{L^{\frac{2(d+2\eta)}{d}}([0,T]\times\mathcal{O})}&=&\left(\int_{0}^{T}\|v^m_s\|^{q}_{L^q}ds\right)^{\frac{1}{q}}\\
&\leq& C\left( \sup_{t\in[0,T]}\|v^m_t\|^2\right)^{\frac{\eta}{d+2\eta}}
\cdot\left(\int_{0}^{T}\|v^m_s\|_{H^{\eta}}^{2}ds\right)^{\frac{d}{2(d+2\eta)}}\nonumber\\
&\leq& C (V^m)^{\frac{\eta}{d+2\eta}}\cdot (V^m)^{\frac{d}{2(d+2\eta)}}=C (V^m)^{\frac{1}{2}}.\nonumber
\end{eqnarray*}
 $(\ref{estimate vm- q})$ and  $(\ref{vm<VM})$ are implied by the first and second inequalities above respectively.
\end{proof}

\vspace{1mm}
The relationship between iteration sequence $V^ms$ is discussed as follows.
\begin{lemma}
\label{iteration}
Assume $v_0$ is upper bounded and set $\lambda_0=2\sup_{\Omega\times\mathcal{O}}v_0$. Suppose that $\tilde f\in L^\infty(\Omega; L^k([0,T]\times\mathcal{O}))$,  $\tilde h,\tilde g \in L^\infty(\Omega; L^{2k}([0,T]\times\mathcal{O}))$, for some integer $k>1+\frac{d}{2\eta}$. Set $\lambda>(\lambda_0\vee 1)$ and $\alpha_0=\frac{2\eta k-d-2\eta}{dk}$. There exists a constant $K>0$ such that for any $m\in \mathbb{N}^+$,
\begin{eqnarray*}
V^m\leq \frac{K^m}{\lambda^{2\alpha}_0} (V^{m-1})^{1+\alpha_0}+K \sup_{t\in[0,T]}|M^m_t|\,.
\end{eqnarray*}
\end{lemma}
\begin{proof}
Recalling the upper bound of $V^m$ in ($\ref{estimate Vm}$), we will estimate its right hand side term by term.

Firstly, set $\beta^{-1}=1-\frac{d}{2(d+2\eta)}-\frac{1}{k}>0$, then H\"{o}lder's inequality implies that
\begin{equation}
\label{<f,v^m>}
\begin{split}
\int_{0}^{T}(\tilde{f}_s, v^m_s)ds
&\leq \|\tilde{f}\|_{L^k([0,T]\times\mathcal{O})}
\cdot \|I_{\{v^m>0\}}\|_{L^\beta([0,T]\times\mathcal{O})}\cdot \|v^m\|_{L^{\frac{2(d+2\eta)}{d}}([0,T]\times\mathcal{O})}\\
&\leq\|\tilde{f}\|_{L^k([0,T]\times\mathcal{O})}
\cdot\left(\int_{0}^{T}\int_{\mathcal{O}}I_{\{v^m_s>0\}}dxds\right)^{\frac{1}{\beta}}\cdot\|v^{m-1}\|_{L^{\frac{2(d+2\eta)}{d}}([0,T]\times\mathcal{O})}\\
&\leq \|\tilde{f}\|_{L^k([0,T]\times\mathcal{O})}
\cdot\left(\int_{0}^{T}\int_{\mathcal{O}}\Big|\frac{2^m v^{m-1}_s}{\lambda}\Big|^{\frac{2(d+2\eta)}{d}}dxds\right)^{\frac{1}{\beta}}
\cdot\|v^{m-1}\|_{L^{\frac{2(d+2\eta)}{d}}([0,T]\times\mathcal{O})}\\
&\leq\left(\frac{2^m}{\lambda}\right)^{\frac{2(d+2\eta)}{d\beta}}\cdot\|\tilde{f}\|_{L^k([0,T]\times\mathcal{O})}
\cdot\|v^{m-1}\|^{1+\frac{2(d+2\eta)}{d\beta}}_{L^{\frac{2(d+2\eta)}{d}}([0,T]\times\mathcal{O})}\\
&\leq C\left(\frac{2^m}{\lambda}\right)^{1+2\alpha_0}\cdot\|\tilde{f}\|_{L^k([0,T]\times\mathcal{O})}
\cdot(V^{m-1})^{1+\alpha_0}, \quad a.s.,
\end{split}\end{equation}
where the last inequality is from $(\ref{vm<VM})$.

Secondly, set $\gamma^{-1}=1-k^{-1}$, it follows that
\begin{equation}\label{hIvm>0}\begin{split}
\int_{0}^{T}\|I_{\{v^m>0\}}\tilde{h}_s\|^2ds
&\leq \|\tilde{h}\|^2_{L^{2k}([0,T]\times\mathcal{O})}
\cdot\left( \int_{0}^{T}\int_{\mathcal{O}}\left(\frac{2^m v^{m-1}_s}{\lambda}\right)^{\frac{2(d+2\eta)}{d}}dxds\right)^{\frac{1}{\gamma}}\\
&\leq \left(\frac{2^m}{\lambda}\right)^{\frac{2(d+2\eta)}{d\gamma}}\cdot \|\tilde{h}\|^2_{L^{2k}([0,T]\times\mathcal{O})}
\cdot \|v^{m-1}\|^{\frac{2(d+2\eta)}{d\gamma}}_{L^{\frac{2(d+2\eta)}{d}}([0,T]\times\mathcal{O})}\\
&\leq C\left(\frac{2^m}{\lambda}\right)^{2+2\alpha_0}\cdot \|\tilde{h}\|^2_{L^{2k}([0,T]\times\mathcal{O})}
\cdot \left(V^{m-1}\right)^{1+\alpha_0}, \quad a.s..
\end{split}\end{equation}

The term $\int_{0}^{T}\|I_{\{v^m>0\}}\tilde{g}_s\|^2ds$ can be estimated similarly.

Finally, note that if $\lambda>\lambda_0$, then $\|v^m_0\|=0$. Combining the inequalities ($\ref{<f,v^m>}$) and ($\ref{hIvm>0}$),  we get the lemma proved.
\end{proof}
The tail probability of $\|v^+\|_{L^{\infty}([0,T]\times\mathcal{O})}$ is estimated in the next lemma.
\begin{lemma}
\label{u+ V0}
Under the assumption in Lemma  \ref{iteration},  and  $\tilde{h}\in L^{2\gamma}([0,T]\times\mathcal{O})$ a.s., where $\frac{1}{\gamma}=\frac{1}{k}-\frac{d\alpha_0}{d+2\eta}$.  Then there is a constant $C'>0$ such that, for any $\lambda>\lambda_0>1$,
\begin{eqnarray*}
P\left(\left\{\|v^+\|_{L^{\infty}([0,T]\times\mathcal{O})}>\lambda, V^{0}\leq \sqrt{\lambda}\right\}\right)\leq 2e^{-C'\lambda^{2\alpha_0}}\,.
\end{eqnarray*}
\end{lemma}

\begin{proof}
Set
$$
S^m=\left\{V^m\leq \frac{\lambda^{2\theta_0}}{\mu^m}\right\},
$$
where the constants $\mu>1$, $\theta_0>0$ will be valued later.

It is easy to check that
\begin{eqnarray*}
\left\{\|v^+\|_{L^{\infty}([0,T]\times\mathcal{O})}>\lambda, V^{0}\leq \lambda^{2\theta}\right\}\subset\bigcup_{m\in \mathbb{N}^+}\left((S^m)^c\cap S^{m-1}\right).\nonumber
\end{eqnarray*}
On $(S^m)^c\cap S^{m-1}$, we have
\begin{eqnarray*}
\sup_{t\in[0,T]}|M^m_t|&\geq& K^{-1}V^m-\frac{K^{m-1}}{\lambda^{2\alpha_0}}(V^{m-1})^{1+\alpha_0}\nonumber\\
&\geq& \frac{\lambda^{2\theta_0}}{K\mu^m}-\frac{K^{m-1}}{\lambda^{2\alpha_0}}\left(\frac{\lambda^{2\theta_0}}{\mu^{m-1}}\right)^{1+\alpha_0}\nonumber\\
&=& \left(\frac{\lambda^{2\theta_0}}{\mu^{m-1}}\right)^{1+\alpha_0}
\left(\frac{\mu^{m\alpha_0-\alpha_0-1}}{K\lambda^{2\alpha_0\theta_0}}-\frac{K^{m-1}}{\lambda^{2\alpha_0}}\right)\,.\nonumber
\end{eqnarray*}
Set
$$
\gamma_m=\left(\frac{\lambda^{2\theta_0}}{\mu^{m-1}}\right)^{1+\alpha_0}\
\mbox{and}\ \
\beta_m=\frac{\mu^{m\alpha_0-\alpha_0-1}}{K\lambda^{2\alpha_0\theta_0}}-\frac{K^{m-1}}{\lambda^{2\alpha_0}}\,.
$$

Let $\theta_0=\frac{1}{2}$,  $\mu=F^{\frac{1}{\alpha_0}}$, where $F$ is a positive number such that $F> K \vee \sqrt{2N}$.  Since $\lambda>\lambda_0$,
\begin{eqnarray*}
\beta_m=\frac{F^{m-1}}{\lambda^{\alpha_0}}\left( \frac{1}{K\mu}-\frac{1}{\lambda^{\alpha_0}}\left(\frac{K}{F}\right)^{m-1}\right)
\geq\frac{F^{m-1}}{\lambda^{\alpha_0}}\left( \frac{1}{K\mu}-\frac{1}{\lambda^{\alpha_0}}\right)
=\tilde{C} \frac{F^{m}}{\lambda^{\alpha_0}}\,.
\end{eqnarray*}

By further calculation, set $\frac{1}{\gamma'}=1-\frac{1}{\gamma}$, then it is easy to check that 
$\frac{d+2\eta}{d\gamma'}=1+2\alpha_0$.

 we obtain that
\begin{eqnarray*}
\langle M^m\rangle_T&=&\int_{0}^{T}\left|( v^m_s, \tilde{h}_s)\right|^2ds\nonumber\\
&\leq& \sup_{t\in[0,T]}\|v^m_t\|^2\cdot\int_{0}^{T}\|I_{\{v^m>0\}}\tilde{h}_s\|^2ds\nonumber\\
&\leq& \sup_{t\in[0,T]}\|v^m_t\|^2\cdot \|\tilde{h}\|^2_{L^{2\gamma}}
\cdot \left( \int_{0}^{T}\int_{\mathcal{O}}\left(\frac{2^m v^{m-1}}{\lambda}\right)^{\frac{2(d+2\eta)}{d}}dxds\right)^{\frac{1}{\gamma'}}\nonumber\\
&\leq& \left(\frac{2^m}{\lambda}\right)^{2+4\alpha_0}\cdot\sup_{t\in[0,T]}\|v^m_t\|^2\cdot \|\tilde{h}\|^2_{L^{2\gamma}}
\cdot \|v^{m-1}\|^{2+4\alpha_0}_{L^{\frac{2(d+2\eta)}{d}}([0,T]\times\mathcal{O})}\nonumber\\
&\leq&C \left(\frac{2^m }{\lambda}\right)^{2+4\alpha_0}\cdot\|\tilde{h}\|^2_{L^{2\gamma}}
\cdot(V^{m-1})^{2+2\alpha_0}\nonumber\\
&\leq& \frac{N^m}{\lambda^{2\alpha_0}}\left(V^{m-1}\right)^{2+2\alpha_0},
\end{eqnarray*}
where the last inequality is obtained by choosing N large enough.
Then,
\begin{eqnarray*}
P\Big((S^m)^c\cap S^{m-1}\Big)&\leq& P\Big(\sup_{t\in[0,T]}M^m_t>\beta_m \gamma_m, \ (V^{m-1})^{1+\alpha_0}\leq \gamma_m\Big)\nonumber\\
&\leq& P\Big(\sup_{t\in[0,T]}M^m_t>\beta_m \gamma_m, \ \langle M^m\rangle_T\leq \frac{N^m}{\lambda^{4\alpha_0}}\gamma^2_m\Big)\nonumber\\
&\leq& e^{-\frac{\beta_m^2 }{2N^m}\lambda^{4\alpha_0}}\leq e^{-\frac{\tilde{C}^2}{2}\frac{F^{2m}}{N^m}\lambda^{2\alpha_0}}\leq e^{-C' 2^m \lambda^{2\alpha_0}}\leq e^{-C'm\lambda^{2\alpha_0}}\,.\nonumber
\end{eqnarray*}

Finally, it follows that
$$
P\left(\left\{\|v^+\|_{L^{\infty}([0,T]\times D)}>\lambda, V^{0}\leq \lambda^{2\theta_0}\right\}\right)\leq \sum_{m\in\mathbb{ N}^+}P\Big((S^m)^c\cap S^{m-1}\Big)\leq 2e^{-C'\lambda^{2\alpha_0}}\,.
$$
\end{proof}
Finally, we come to prove the $L^p-$estimates for the time-space uniform norm of weak solutions.
\begin{theorem}
Under the conditions in Lemma \ref{u+ V0},  for any $p>2$, we assume that $\tilde{f}, \tilde{g}, \tilde{h}\in L^{2q}(\Omega;L^2([0,T]\times\mathcal{O}))$ and $q>2p$, then
\begin{eqnarray*}
E\|(u-S')^+\|^p_{L^{\infty}([0,T]\times\mathcal{O})}<+\infty.
\end{eqnarray*}
\end{theorem}

\begin{proof}
Recall the estimate ($\ref{vm estimates}$) and let $m=0$. We obtain that
\begin{eqnarray*}
&&\sup_{t\in[0,T]}\|v^0_t\|^2+\int_{0}^{T}\|v^0_s\|^2_{H^\eta}+\sum_{k=1}^{n}\|L_k v^0_s\|^2ds\\
&\leq& \|v^0_0\|^2+2\int_{0}^{T}\|\tilde{f}_s I_{\{v^0_s>0\}}\|^2+\|\tilde{g}_s I_{\{v^0_s>0\}}\|^2+\|I_{\{v^0_s>0\}}\tilde{h}_s\|^2ds
+2\sup_{t\in[0,T]}|M^0_t|\,.\nonumber
\end{eqnarray*}

By Doob's martingale inequality,
\begin{eqnarray}
E\left[\sup_{t\in[0,T]}|M^0_t|^q\right]&\leq& CE\left[\langle M^0\rangle_T^{\frac{q}{2}}\right]\nonumber\\
&\leq& C'E\left[\left(\int_{0}^{T}(v^0_s, \tilde{h}_sI_{\{v^0_s>0\}})^2ds\right)^{\frac{q}{2}}\right]\nonumber\\
&\leq& C'E\left[\left(\sup_{t\in[0,T]}\|v^0_t\|^q\right)\cdot\left(\int_{0}^{T}\|\tilde{h}_s I_{\{v^0_s>0\}}\|^2ds\right)^{\frac{q}{2}}\right]\nonumber\\
&\leq& C'\epsilon E\left[\sup_{t\in[0,T]}\|v^0_t\|^{2q}\right]+\frac{C'}{\epsilon}E\left[\left(\int_{0}^{T}\|\tilde{h}_s I_{\{v^0_s>0\}}\|^2ds\right)^q\right].\nonumber
\end{eqnarray}
Choosing $\epsilon$ small enough, then it follows that
\begin{eqnarray*}
&&E\left[\sup_{t\in[0,T]}\|v^0_t\|^{2q}+\left(\int_{0}^{T}\|v^0_s\|^2_{H^\eta}+\sum_{k=1}^{n}\|L_k v^0_s\|^2ds\right)^q\right]\nonumber\\
&\leq& C'' E\left[\|v^0_0\|^{2q}+\|\tilde{f} I_{\{v^0>0\}}\|^{2q}_{L^{2}([0,T]\times\mathcal{O})}+\|\tilde{h} I_{\{v^0>0\}}\|^{2q}_{L^{2}([0,T]\times\mathcal{O})}
+\|\tilde{g} I_{\{v^0>0\}}\|^{2q}_{L^{2}([0,T]\times \mathcal{O})}\right].
\end{eqnarray*}

Then
\begin{eqnarray*}
&&E(V^0)^{q}\leq C E\left[\sup_{t\in[0,T]}\|v^0_t\|^{2q}+\left(\int_{0}^{T}\|v^0_s\|^2_{H^\eta}+\sum_{k=1}^{n}\|L_k v^0_s\|^2ds\right)^q\right]\nonumber\\
&\leq& C E\left[\|v^0_0\|^{2q}+\|\tilde{f} I_{\{v^0>0\}}\|^{2q}_{L^{2}([0,T]\times\mathcal{O})}+\|\tilde{h} I_{\{v^0>0\}}\|^{2q}_{L^{2}([0,T]\times\mathcal{O})}
+\|\tilde{g}I_{\{v^0>0\}}\|^{2q}_{L^{2}([0,T]\times\mathcal{O})}\right].
\end{eqnarray*}

Finally, we obtain the $L^p-$estimate for the solution of OSPDEs:
\begin{eqnarray}
&&E\|(u-S')^+\|^p_{L^{\infty}([0,T]\times\mathcal{O})}\nonumber\\
&=&p\int_{0}^{\infty} P\left(\left\{\|v^+\|_{L^{\infty}([0,T]\times\mathcal{O})}>\lambda\right\}\right)\lambda^{p-1}d\lambda\nonumber\\
&\leq& p\int_{0}^{\lambda_0}\lambda^{p-1}d\lambda
+p\int_{\lambda_0}^{\infty} P\left(\left\{\|v^+\|_{L^{\infty}([0,T]\times\mathcal{O})}>\lambda\right\}\right)\lambda^{p-1}d\lambda\nonumber\\
&\leq& \lambda^p_0
+p\int_{\lambda_0}^{\infty} P\left(\left\{\|v^+\|_{L^{\infty}([0,T]\times\mathcal{O})}>\lambda , V^0\leq \sqrt{\lambda}\right\}\right)\lambda^{p-1}d\lambda\nonumber\\
&&+p\int_{\lambda_0}^{\infty} P\left(\left\{ V^0> \sqrt{\lambda}\right\}\right)\lambda^{p-1}d\lambda\nonumber\\
&\leq&\lambda^p_0+2p\int_{\lambda_0}^{\infty}e^{-C'\lambda^{2\alpha_0}}\lambda^{p-1}d\lambda
+p\int_{\lambda_0}^{\infty}\frac{E[|V^0|^q]}{\lambda^{\frac{q}{2}}}\lambda^{p-1}d\lambda\nonumber\\
&<&+\infty\,,\qquad\quad\mbox{when} \quad q>2p.\nonumber
\end{eqnarray}

\end{proof}

\textbf{Acknowlegement}

The authors would like to thank two referees for the careful reading and very useful comments.

\end{document}